\documentclass[10pt, a4paper]{amsart}
\usepackage{amssymb,amsmath,amsthm,latexsym,amscd,amsfonts,mathrsfs}
\usepackage{graphicx}
\usepackage[utf8]{inputenc}
\newcommand{\im}{\text{\rm Im\,}}

\newcommand{\bd}{{\mathbb{D}}}
\newcommand{\bn}{{\mathbb{N}}}

\newcommand{\bz}{{\mathbb{Z}}}
\newcommand{\bc}{{\mathbb{C}}}

\newcommand{\bt}{{\mathbb{T}}}

\newcommand{\css}{{\mathcal{S}}}
\newcommand{\ccs}{{\mathcal{C}}}

\renewcommand{\l}{\lambda}
\newcommand{\s}{\sigma}

\renewcommand{\ss}{\Sigma}

\newcommand{\p}{\varphi}

\renewcommand{\th}{\theta}
\renewcommand{\d}{\delta}
\renewcommand{\dh}{{\hat\delta}}

\newcommand{\dd}{\Delta}
\renewcommand{\o}{\omega}

\newcommand{\g}{\gamma}
\newcommand{\gi}{{1/\gamma}}
\newcommand{\go}{{\gamma/(1+\gamma)}}
\newcommand{\gga}{\Gamma}
\newcommand{\ep}{\epsilon}

\newcommand{\nno}{{n=1,\dots,\infty}}
\newcommand{\tn}{{\tilde n}}
\newcommand{\po}{{\p_0}}
\newcommand{\pon}{{\p_1}}

\newcommand{\sg}{{\Sigma_\g}}
\newcommand{\sgo}{{\Sigma^0_\g}}
\newcommand{\ddg}{{\Delta_\g}}
\newcommand{\dg}{{\delta_\g}}
\newcommand{\tc}{{\tilde\chi}}
\newcommand{\tcso}{{\tilde\chi_{S_1}}}
\newcommand{\tcsl}{{\tilde\chi_{S_L}}}

\newcommand{\shl}{{\hat S_L}}
\newcommand{\sho}{{\hat S_1}}
\newcommand{\cho}{{\chi_{00}}}
\newcommand{\chn}{{\chi_{01}}}
\newcommand{\tpo}{{T_{\p^{1/2}_0}}}
\newcommand{\tpos}{{T^*_{\p^{1/2}_0}}}
\newcommand{\tpn}{{T_{\p_1}}}

\newcommand{\nn}{\nonumber}
\newcommand{\nt}{\noindent}

\newcommand{\pt}{\partial}
\newcommand{\ti}{\tilde}

\newcommand{\lp}{\left(}
\newcommand{\lpp}{{\big(}}
\newcommand{\rp}{\right)}
\newcommand{\rpp}{{\big)}}
\newcommand{\lt}{\left}
\newcommand{\rt}{\right}

\newcommand{\lsup}{{\limsup_{s\to 0+}\, }}
\newcommand{\linf}{{\liminf_{s\to0+}\, }}
\newcommand{\lni}{{L^\infty(\bt)}}

\newcommand{\diag}{{\mathrm{diag}\, }}

\renewcommand{\css}{\mathcal{S}}

\newcommand{\ie}{{\emph{i.e., }}}

\allowdisplaybreaks
\numberwithin{equation}{section}

\newtheorem{theorem}{Theorem}[section]

\newtheorem{lemma}[theorem]{Lemma}
\newtheorem{corollary}[theorem]{Corollary}
\newtheorem{proposition}[theorem]{Proposition}

\begin{document}

\title[Toeplitz operators on Bergman space]
{On spectral properties of compact Toeplitz operators on Bergman space with logarithmically decaying symbol  and applications to banded matrices}

\author{M. Koïta}
\address{Institut Universitaire de Formation Professionnelle, Université de Ségou, B.P. 97, Ségou, Mali}
\email{kmahamet@yahoo.fr}

\author{S. Kupin}
\address{Institut de Mathématiques de Bordeaux UMR5251, CNRS, Université de Bordeaux, 351 ave. de la Libération, 33405 Talence Cedex, France}
\email{skupin@math.u-bordeaux.fr}

\author{S. Naboko}
\address{Physics Institute, St. Petersburg State University, Ulyanovskaya str. 1, St. Peterhof, St. Petersburg, 198-904 Russia}
\email{sergey.naboko@gmail.com}

\author{B. Touré}
\address{Institut Universitaire de Formation Professionnelle, Université de Ségou, B.P. 97, Ségou, Mali}
\address{Faculté des Sciences et  des Techniques, Université des Sciences, des Techniques et des Technologies de Bamako, Campus Universitaire de Badalabougou à Bamako, B.P. E-3206, Bamako, Mali}
\email{vbelco@yahoo.fr}

%\date{March, 15, 2020}

\begin{abstract}
Let $L^2(\bd)$ be the space of measurable square-summable functions on the unit disk. Let $L^2_a(\bd)$ be the Bergman space, \ie the (closed) subspace of analytic functions in $L^2(\bd)$. $P_+$ stays for the orthogonal projection going from $L^2(\bd)$ to $L^2_a(\bd)$. For a function $\p\in L^\infty(\bd)$, the Toeplitz operator $T_\p: L^2_a(\bd)\to L^2_a(\bd)$ is defined as
$$
T_\p f=P_+\p f, \quad f\in L^2_a(\bd).
$$
The main result of this article are spectral asymptotics for singular (or eigen-) values of compact Toeplitz operators with logarithmically decaying symbols, that is
$$
\p(z)=\p_1(e^{i\th})\, (1+\log(1/(1-r)))^{-\g},\quad \g>0,
$$
where $z=re^{i\th}$ and $\p_1$ is a continuous (or piece-wise continuous) function on the unit circle. The result is applied to the spectral analysis of banded (including Jacobi) matrices.
\end{abstract}


\maketitle

\vspace{-0.7cm}
\section*{Introduction}\label{s0}
The investigation of Toeplitz operators is an important topic of modern analysis. The theory of Toeplitz operators on Hardy spaces was developed extensively in 70-90's of the last century, see monographs by Nikolski \cite{nk1,nk2} and Böttcher-Silbermann \cite{bo1} for a detailed account on the topic. The study of Toeplitz operators on larger functional spaces (\ie Fock, Bergman spaces, etc.) is currently under progress. 

In this article, we are interested in spectral asymptotics of singular values of a compact Toeplitz operator on the Bergman space with logarithmically decaying symbol. We need several definitions to give the formulations of our results.

Let $\bd:=\{z:|z|<1\}$ and $\bt:=\pt\bd=\{z : |z|=1\}$ be the unit disk and the unit circle, respectively. As usual, we denote by $L^2(\bd)=L^2(\bd, dA)$ the space of measurable square-summable functions on $\bd$ with respect to the normalized Lebesgue measure $dA(z):=dxdy/\pi, \ z\in\bd$. The Bergman space is defined as a (closed) subspace of analytic functions lying in $L^2(\bd)$, that is
\begin{equation}\label{e0011}
L^2_a(\bd):=L^2_a(\bd,dA)=\{f\in\mathcal{A}(\bd): ||f||^2_2=\int_\bd |f(z)|^2\, dA(z)<\infty\}.
\end{equation}
The Riesz orthogonal projection $P_+: L^2(\bd)\to L^2_a(\bd)$ is given by the integral operator
\begin{equation}\label{e0012}
(P_+f)(z):=\int_\bd \frac 1{(1-z\bar w)^2} f(w)\, dA(w),\quad f \in L^2(\bd).
\end{equation}
A detailed treatment of these and other objects pertaining to Bergman spaces can be found in Hedenmalm-Korenblum-Zhu \cite{hkz}, Zhu \cite[Chap. 4]{zh1}.

For a symbol $\p\in L^\infty(\bd)$, the corresponding Toeplitz operator is defined by the relation
\begin{equation}\label{e0013}
T_\p f:=P_+\p f, \quad f\in L^2_a(\bd).
\end{equation}
Many analytic  properties of these operators are very-well understood, see Zhu \cite[Chap. 7]{zh1} for a nice presentation of the subject. %for a complete and pleasant presentation of the subject. 
For instance, we have trivailly
$$
||T_\p||\le ||\p||_\infty,
$$
so the Toeplitz operator corresponding to a bounded symbol is also bounded. The compactness of $T_\p$ is related to the behavior of the symbol $\p$ on the vicinity of the unit circle $\bt$. The following proposition is not difficult to prove.
\begin{proposition}[{\cite[Prop. 7.3]{zh1}}]\label{p001} Let $\p\in C(\bar \bd)$, the class of contionuous functions on the closure of $\bd$. Then, $T_\p$ is compact if and only if
$$
\lim_{|z|\to 1-}\p(z)=0, \quad z\in\bd.
$$
\end{proposition}
For a positive symbol $\p$, criteria for $T_\p$ to be compact or to belong to Schatten-von Neumann classes $\css_p, \ 0<p<\infty$, can be found in \cite[Sect. 7.3]{zh1}.

Subsequent results in this direction address the spectral asymptotics (\ie the asymptotics of eigen- and/or singular values) of Toeplitz operators with symbols from some special classes.

Consider a symbol $\p$ of the following form
\begin{equation}\label{e003}
\p(z):=\p_1(e^{i\th})\p_0(r)=\p_1(e^{i\th})\, (1-r)^{\g},
\end{equation}
where $\p_1\in L^{1/\g}(\bt), \ \g>0$, and $z:=re^{i\th}\in\bd$. By Proposition \ref{p001}, the Toeplitz operator $T_\p$ is obviously compact. Its singular values $\{s_n(T_\p)\}_n$ form a decreasing sequence and converge to zero. 
An important result to come was obtained in Pushnitski \cite{pu1}.
\begin{theorem}[{\cite[Thm. 1.1]{pu1}}]\label{t001}
Consider the Toeplitz operator $T_\p$ with symbol defined in \eqref{e003}. Its singular values $\{s_n(T_\p)\}_n$ have the following asymptotics
\begin{eqnarray}\label{e004}
\lim_{n\to+\infty} n^\g s_n(T_\p)&=&\frac{\gga(\g+1)}{2^\g}||\p_1||_{L^\gi(\bt)}\\
&:=&\frac{\gga(\g+1)}{2^\g}\lp\int^{2\pi}_0 |\p_1(e^{i\th})|^\gi\, \frac{d\th}{2\pi}\rp^\g. \nn
\end{eqnarray}
%where the norm in the RHS of the latter relation is taken in $L^\gi(\bt)$. 
Above, $\gga(.)$ is Euler gamma-function.
\end{theorem}
The above result was obtained in \cite{pu1} in slightly more general form. Defining the counting function $n(.,T_\p)$ for the sequence of singular numbers of $T_\p$ as 
$$
n(s,T_\p):=\#\{n: s_n(T_\p)>s\}, \ s>0, 
$$
we can rewrite \eqref{e004} in an equivalent manner
$$
\lim_{s\to0+} s^\gi n(s, T_\p)=\frac{\gga(\g+1)^\gi}2\, ||\p_1||^\gi_{L^\gi(\bt)}.
$$
The proof of the above theorem is purely operator-theoretic and it uses some basic facts on underlying Bergman space only. We also mention a recent article El-Fallah-El-Ibbaoui \cite{fa} on a closely related topic.

The main result of the present article is a counterpart of Theorem \ref{t001} in logarithmic scale.
Let $\p_1\in C(\bt)$, the space of continuous functions on the unit circle. For a fixed $\g>0$, consider
\begin{equation}\label{e0041}
\p(z):=\p_1(e^{i\th})\p_0(r),
\end{equation}
where 
\begin{equation}\label{e0042}
\p_0(r):=\p_{0,\g}(r)=\frac 1{\lp1+\log\frac1{(1-r)}\rp^\g}, \qquad r\in [0,1).
\end{equation}
That is, the symbol $\p(z)$ logarithmically tends to zero when $z$ goes to the unit circle $\bt$.

\begin{theorem}\label{t1} Let $\p$ be the a symbol defined above, $\g>0$. The following asymptotics hold for the singular values of Toeplitz operator $T_\p$
\begin{equation}\label{e005}
\lim_{n\to+\infty} (\log(n+1))^{\g}s_n(T_\g)=||\p_1||_{L^\infty(\bt)}.
\end{equation}
\end{theorem}
Generalizations of the above theorem are given in Section \ref{s51}. 

Recalling the definition of the counting function, we can rewrite \eqref{e005} as
\begin{equation}\label{e0051}
\lim_{s\to 0+} s^{1/\g}\log n(s,T_\p)=||\p_1||^{1/\g}_{L^\infty(\bt)},
\end{equation}
where $\p_1\not\equiv 0$ on $\bt$.
We stress that despite a partial similarity, certain important pieces of the proof of Theorem \ref{t1} seem to be more involved than those of Theorem \ref{t001}. Among other technical points, it relies essentially on the structure of Bergman space and it uses fine properties of the Riesz orthoprojector \eqref{e0012} on it.

The organization of the article is rather straightforward. A fast sample of the theory of compact operators with logarithmically decaying singular values is developed in Section \ref{s2}. This section also contains a result on asymptotic orthogonality of certain operators which will be the cornerstone for the proof of  Theorem \ref{t1}. It is proved in Section \ref{s3}. The asymptotic orthogonality for a specific family of Toeplitz operators is obtained in Section \ref{s4}. Section \ref{s5} gives slightly different versions the obtained results as well as their applications to the study of spectral properties of compact banded matrices with logarithmically decaying entries.

Throughout the article, ``generic'' constants change from one relation to another. Significant constants are sub-indexed like $C_1, C_2$, etc. Points of the unit disk $\bd$ are often written as $z=|z|e^{i\th}:=re^{i\th}, \
r\in[0,1), \th\in [0,2\pi)$. As usual, the unit circle $\bt=\{z=e^{i\th}: \th\in[0,2\pi)\}$ is identified with the interval $[0,2\pi]$.

\section{Some special classes of compact operators}\label{s2}

\subsection{Basic notions and Schatten-von Neumann classes of compact ope\-ra\-tors}\label{s21}
In the paper, we shall use a number of basic facts on compact operators, see Gohberg-Krein \cite[Chap. 2]{gk1} and Birman-Solomyak \cite[Chap. 11]{bs1}.
 %The monographs contain a complete presentation of the topic.

Let $A$ be a compact operator on a Hilbert space. The class of compact operators on the space is denoted by $\css_\infty$. It is well-known that the spectrum $\s(A_0)$ of a self-adjoint compact operator $A_0$ consists of  the closure of the set of eignevalues $\{\l_n(A_0)\}_\nno$ tending to zero. The singular values of a compact operator $A$ are defined as
$$
s_n(A):=\l_n(A^*A)^{1/2}, \quad n\ge 1.
$$
The sequence of singular values $\{s_n(A)\}_\nno$ is written taking into account the multiplicities; it is positive, decreasing and $s_n(A)$ goes to zero as $n\to +\infty$. It is a simple fact that for $A\in\css_\infty$ and a bounded operator $B$, one has
\begin{equation}\label{e02}
s_n(BA)\le ||B||s_n(A),
\end{equation}
see \cite[Chap. 11, Sect. 1]{bs1} or \cite[Chap. 2, Sect. 2]{gk1}.
Another important characteristic connected to the sequence   
$\{s_n(A)\}_n$ is the counting function,
\begin{equation}\label{e01}
n(s,A):=\#\{n: s_n(A)>s\}, \quad s>0.
\end{equation}
For instance, for $A,B\in \css_\infty$ and $s>0$, one has
\begin{equation}\label{e61}
n(s,AB)\le n(s_1)+n(s_2, B),
\end{equation}
where $s=s_1s_2$ and $s_1,s_2>0$, see \cite[Chap. 11, Sect. 1]{bs1} once again.

The Schatten-von Neumann classes $\css_p, \ 0<p<\infty$, are given by
$$
\css_p:=\{A\in \css_\infty: ||A||^p_{\css_p}=\sum^\infty_{n=1} s_n(A)^p<\infty\}, \quad A\in \css_p \Leftrightarrow \int^\infty_0 s^{p-1} n(s,A)\, ds<\infty.
$$
The class $\css_2 \ (p=2)$ is called Hilbert-Schmidt class.

\subsection{Some specific classes of compact operators}\label{s22}
In this paper, we are mainly interested in classes of compact operators with logarithmically decaying singular values. 
Below, we introduce the definitions and give the properties of operators from these classes. Up to certain technical aspects, the proofs of the assertions of this subsection follow Birman-Solomyak \cite[Chap. 11, Sect. 6]{bs1} and they are therefore omitted.
%To deal with these classes, we introduce appropriate definitions as well as certain modifications to classical notions presented in the previous subsection.

To start with, we consider a ``shifted'' counting function
\begin{equation}\label{e1}
\tn(s,A):=n(s, A)+2, 
\end{equation}
compare with \eqref{e01}. We shall see that $\tn(.,A)$ is more convenient for our purposes from technical point of view. For $\g>0$, set
\begin{eqnarray}
\sg&:=&\lt\{A\in \css_\infty: s_n(A)=O\lp\frac 1{(\log(n+1))^\g}\rp\rt\} \label{e2}\\
&=&\lt\{A\in \css_\infty: \sup_{n\ge 1} \ (\log n)^\g s_n(A)<+\infty\rt\}, \nonumber \\
\sgo&:=&\lt\{A\in \css_\infty: s_n(A)=o\lp\frac 1{(\log(n+1))^\g}\rp\rt\} \label{e3}\\
&=&\lt\{A\in \css_\infty: \lim_{n\to+\infty} (\log n)^\g s_n(A)=0 \rt\}. \nonumber
\end{eqnarray}

For a compact operator $A$, the following relations are equivalent in a trivial way
\begin{equation}\label{e30}
\sup_{n\ge 1}\ (\log (n+1))^\g s_n(A)<\infty, \quad \sup_{s>0}\ s^\gi \log \tn(s,A)<\infty.
\end{equation}
Moreover, the equality
\begin{equation}\label{e301}
s_n(A)=\frac{C}{(\log (n+1))^\g}(1+o(1)), \ n\to+\infty,
\end{equation}
is equivalent to
\begin{equation}\label{e302}
\log\tn(s, A)=\frac{C^\gi}{s^\gi}(1+o(1)), \ s\to0+.
\end{equation}
Two remarks are in order. First,  one can work similarly with the usual and the ``shifted'' counting functions $n(.,A)$ and $\tn(.,A)$, respectively. Second, the classes $\sg$ and $\sgo$ can be similarly defined as
\begin{eqnarray*}
\sg&:=&\lt\{A\in \css_\infty: \log \tn(s,A)=O\lp\frac 1{s^\gi}\rp, \ s>0\rt\} \\ %\label{e21}
&=&\lt\{A\in \css_\infty: \sup_{s>0} s^\gi \log\tn(s,A)<+\infty\rt\},\\
\sgo&:=&\lt\{A\in \css_\infty: \log \tn(s,A)=o\lp\frac 1{s^\gi}\rp, \ s>0\rt\} \\ %\label{e31}\\
&=&\lt\{A\in \css_\infty: \lim_{s\to 0+} s^\gi \log \tn(s,A)=0 \rt\}. 
\end{eqnarray*}

Once again, let $A\in \sg$. The next quantities will be useful in the sequel.
\begin{equation}\label{e4}
\ddg(A)=\limsup_{s\to 0+} s^\gi \log\tn(s,A),\qquad \dg(A)=\liminf_{s\to 0+} s^\gi \log\tn(s,A).
\end{equation}

\begin{proposition}\label{p1}
Let $A,B\in \sg$. Then
\begin{eqnarray}
&&\lpp\ddg(A+B)\rpp^\go\le \lpp\ddg(A)\rpp^\go+\lpp\ddg(B)\rpp^\go, \label{e51}\\
&&|\lpp\ddg(A)\rpp^\go-\lpp\ddg(B)\rpp^\go|\le \lpp\ddg(A-B)\rpp^\go, \label{e52}\\  
&&|\lpp\dg(A)\rpp^\go-\lpp\dg(B)\rpp^\go|\le\lpp\ddg(A-B)\rpp^\go.\label{e53}
\end{eqnarray}
\end{proposition}
The proof of the above proposition follows \cite[Chap. 11, Sect. 6, Thm. 4, Cor. 5]{bs1}. The next proposition is its simple corollary.
\begin{proposition}[Ky-Fan-type lemma]\label{p2}
Let $A\in\sg$ and $B\in\sgo$. Then
\begin{equation}\label{e6}
\ddg(A+B)=\ddg(A),\qquad  \dg(A+B)=\dg(A).
\end{equation}
\end{proposition}

\begin{proposition}\label{p3} We have:
\begin{enumerate}
\item for any $p,\g>0$, $\css_p\subset \ss^0_\g$,
\item if $A,B\in \sg$, then $\dd_{2\g}(AB)\le\ddg(A)+\ddg(B)$,
and, consequently, $AB\in \ss_{2\g}$.
\item if $A\in \sg$ and $B\in\sgo$, then $\dd_{2\g}(AB)=0$, \ie $AB\in \ss^0_{2\g}$.
\end{enumerate}
\end{proposition}

\subsection{Asymptotical orthogonality of a set of operators}\label{s23}
Let $A$ be a compact operator and
$$
A=\sum^L_{k=1} A_k.
$$
Suppose that $A_k\in\css_\infty, \ k=1,\dots,L$, as well. The coming abstract proposition will prove to be a useful tool for our purposes. Up to some technical details, it is is similar to Pushnitski \cite[Theorem 2.2]{pu1}, and we give a fast sketch of its proof only.
\begin{proposition}\label{p4}
Suppose that $A, A_k, \ k=1,\dots, L$, are as above and the family $\{A_k\}_k$ is asymptotically orthogonal, that is
\begin{equation}\label{e08}
A^*_kA_j, A_kA^*_j\in \ss^0_{2\g}, \quad j\not=k,\ \  j,k=1,\dots,L.
\end{equation}
Then 
\begin{equation}\label{e8}
\ddg(A)=\lsup s^\gi\log\lp\sum_{k=1}^L\tn(s,A_k)\rp,\quad \dg(A)=\linf s^\gi\log\lp\sum_{k=1}^L\tn(s,A_k)\rp.
\end{equation}
\end{proposition}

\begin{proof} Set $H_L=\oplus_{k=1}^L H$, and
$$
A_0=\diag\{A_1,\dots, A_L\}: H_L\to H_L,
$$
that is 
$$A_0(f_1,\dots, f_L)=(A_1f_1,\dots A_Lf_L)
$$ 
for arbitrary $(f_1,\dots f_L)\in H_L$.
Consider also the embedding operator $J:H_L\to H$ given by
$$
J(f_1,\dots f_L)=f_1+\dots+f_L.
$$
We have $J^*f=(f,\dots,f):H\to H_L$. A straightforward computation shows that
\begin{eqnarray}
(JA_0)(JA_0)^*f&=&(A_1A_1^*+\dots+A_LA_L^*)f, \label{e91}\\
&&\nn \\
(JA_0)(JA_0)^*&=&
\begin{bmatrix}
A_1^*A_1&\dots &A^*_1A_L\\
\vdots& \ddots &\vdots \\
A^*_LA_1& \dots &A^*_LA_L
\end{bmatrix}, \label{e92}
\end{eqnarray}
where we used a natural block decomposition for the operator $(JA_0)(JA_0)^*: H_L\to H_L$.
Since the operator $A_0$ is block-diagonal on $H_L$, we see
$$
n(s,A_0)=\sum^L_{k=0}n(s, A_k), \ s>0.
$$
Now, by assumption \eqref{e08}, we have substracting the diagonal parts
$$
(JA_0)^*(JA_0)-A^*_0A_0\in \ss^0_{2\g}.
$$
Notice that the singular values of $T^*T$ and $TT^*$ coincide for any compact ope\-ra\-tor $T$, and, in particular,  $\dd_{2\g}((JA_0)^*(JA_0))=\dd_{2\g}((JA_0)(JA_0)^*)$. Hence, by Proposition \ref{p2},
\begin{eqnarray*}
\dd_{2\g}((JA_0)^*(JA_0))&=&\dd_{2\g}(A_0^*A_0)=\ddg(A_0)\\
&=&\lsup s^\gi \log\lp\sum^L_{k=1}n(s,A_k)+2\rp.
\end{eqnarray*}
Furthermore, since $AA^*=\sum^L_{k, j=1}A_kA_j^*$, relation \eqref{e08} and Proposition \ref{p2} yield once again
$$
AA^*-(JA_0)(JA_0)^*\in \ss^0_{2\g},
$$
and $\dd_{2\g}(AA^*)=\dd_{2\g}((JA_0)(JA_0)^*)$. 
Putting together the above computations, we obtain
\begin{eqnarray}%\label{}
\ddg(A)&=&\dd_{2\g}(AA^*)=\dd_{2\g}((JA_0)(JA_0)^*)=\dd_{2\g}((JA_0)^*(JA_0)) \nn\\
&=&\ddg(A_0)=\lsup s^\gi \log\lp\sum^L_{k=1}n(s,A_k)+2\rp. \label{e93}
\end{eqnarray}
Passing from counting functions $n(.,A_k)$ to $\tn(.,A_k)$ is also obvious, and so the first relation in \eqref{e8} is proved. The proof of the second relation in \eqref{e8} is the same (with $\limsup$ replaced by $\liminf$). The proposition is completed.
\end{proof}

\section{Proof of the main theorem}\label{s3}

\subsection{Some notation and starting remarks}
Recall the definition of function $\p$, see \eqref{e0041}, \eqref{e0042}. To begin with, we consider the simplest radial case $\p_1(e^{i\th})\equiv 1$, so that $\p(z):=\p_0(z)=\p_0(r)$.
\begin{lemma}\label{l02} We have
\begin{equation}\label{e102}
\lim_{n\to+\infty} (\log(n+1))^\g s_n(T_{\p_0})=||\p_1||_{L^\infty(\bt)}=1.
\end{equation}
\end{lemma}
\begin{proof}
Since the symbol $\p_0$ is radial,  the matrix of the positive compact ope\-ra\-tor $T_{\p_0}$ is diagonal when computed in the standard orthonormal basis of the Bergman space $\{e_n\}_{n=0,\dots},\ e_n(z)=\sqrt{n+1}\, z^n$. Passing to polar coordinates and using Lemma \ref{l01} (with $g(r)\equiv 1$ on $[0,1]$), we obtain
\begin{equation}\label{e002}
\begin{split}
s_n(T_\po)&=(T_\po e_n, e_n)=2(n+1)\int^1_0 r^{2n+1}\po(r)\,dr\\
&=2(n+1)\, \frac 1{(2n+1)(\log(2n+1))^\g}(1+o(1))\\
&=\frac1{(\log (n+1))^\g}\lp 1+o(1)\rp, \quad n\to+\infty.
\end{split}
\end{equation}
\end{proof}

Now, take a natural $L>0$ and let $I_j:=[2\pi (j-1)/L, 2\pi j/L), \ j=1,\dots, L$ be the partition of $[0,2\pi)$ into disjoint intervals of equal length. Set $\chi_{I_j}$ to be the characteristic functions of the intervals $I_j, \ j=1,\dots, L$, and, moreover
\begin{equation}\label{e101}
\tc_j:=\chi_{I_j}\p_0, \quad T_{\tc_j}:=P_+\tc_j=P_+(\chi_{I_j}\p_0), % \quad j=1,\dots,L.
\end{equation}
see \eqref{e0042}.
It is clear that the operators $T_{\tc_j}$ are unitarily equivalent to $T_{\tc_1}$ by rotation $z\mapsto e^{-i(2\pi (j-1)/L)}z, \ z\in\bd$, and so their singular values and counting functions coincide, $\tn(s,T_{\tc_j})=\tn(s,T_{\tc_1}), \ s>0, \ j=1,\ldots, L$.

\subsection{Proof of Theorem \ref{t1}}\label{s32}
The main tools of the proofs appearing in the present subsection are Proposition \ref{p4} and Theorem \ref{t20} saying that the above defined operators $T_{\tc_j}, \ j=1,\ldots, L$, form an asymptotically orthogonal family. The proof of Theorem \ref{t20} is given in Section \ref{s4}. Recall that $\p=\p_1\p_0$, see \eqref{e0041}, \eqref{e0042}.

\begin{lemma}\label{l1} The following claims hold true:
\begin{enumerate}
\item Let $\p:=\p_0$, so that $\p_1\equiv 1$ on $\bt$. Then
$$
\ddg(T_{\p_0})=\dg(T_{\p_0})=1.
$$
\item For $j=1,\ldots, L$,
\begin{equation*}
\ddg(T_{\tc_j})=\dg(T_{\tc_j})=1.
\end{equation*}
\end{enumerate}
\end{lemma}
\begin{proof}  Theorem \ref{t20} is crucial for the proof of the current lemma. Its first claim is a simple rewriting of relation \eqref{e102} with the help of observation on the equivalence of \eqref{e301} and \eqref{e302}.

Turning to the second claim of the lemma, we keep  $\p:=\p_0$ (or $\p_1\equiv 1$ on $\bt$). 
Of course,
$$
\p_1:=\sum_{j=1}^L\, 1\cdot\chi_j, \quad T_\p:=T_{\p_0}=\sum_{j=1}^L\,1\cdot T_{\tc_j}.
$$
By Theorem \ref{t20}, the operators $T^*_{\tc_j} T_{\tc_k}, T_{\tc_j} T^*_{\tc_k}$ lie in $\ss^0_{2\g}$ for $j\not=k, \ j,k=1,\dots, L$, and so Proposition \ref{p4} shows that
\begin{eqnarray}\label{e10}
\ddg(T_{\p_0})&=&\lsup s^\gi\log\lp\sum^L_{j=1}\tn(s,T_{\tc_j})\rp\\
&=&\lsup s^\gi\log(L\tn(s,T_{\tc_1}))=\ddg(T_{\tc_1}). \nn
\end{eqnarray}
Similarly, one sees $\dg(T_{\p_0})=\dg(T_{\tc_1})$. It remains to recall that counting functions $\tn(., T_{\tc_j})$ and $\tn(., T_{\tc_1})$ coincide, and the lemma is proved.
\end{proof}

The next proposition says that the claim of Theorem \ref{t1} is rather simple to prove when $\p_1$ is a step (or a ``staircase'') function given by
$$
\p_1=\sum^L_{j=1} c_j\cdot \chi_j,
$$
where $c_j\in\bc, \ j=1,\ldots, L$.
\begin{proposition}\label{p41} For a step function $\p_1$ defined above, we have
$$
\ddg(T_\p)=\dg(T_\p)=||\p_1||^\gi_\lni.
$$
\end{proposition}
\begin{proof}
To ease the writing, we suppose that $c_1:=||\p_1||_\lni=\max_{j=1,\dots,L} |c_j|\ge 0$. The case $c_1=0$ being trivial, we assume that $|c_1|>0$. In particular,
$$
\tn(s, c_j T_{\tc_j})\le \tn(s, c_1 T_{\tc_1}),\ j=1,\dots, L,
$$
due to $|c_jT_{\tc_j}|\le |c_1|T_{\tc_j}$.
Applying Theorem \ref{t20} and Proposition \ref{p4} as in \eqref{e10}, we have
\begin{eqnarray*}
\ddg(T_\p)&=&\lsup s^\gi\log\lp\sum^L_{j=1}\tn(s,c_j T_{\tc_j})\rp\\
&=&\lsup s^\gi \log\lp \tn(s,c_1T_{\tc_1})+\sum^L_{j=2} \tn(s, c_jT_{\tc_j})\rp\\
&=&\lsup s^\gi \log \tn(s,c_1T_{\tc_1})\left\{1+\frac{\sum^L_{j=2} \tn(s, c_jT_{\tc_j})}{\tn(s,c_1T_{\tc_1})}\right\}.
\end{eqnarray*}
The term in the figure brackets on the RHS of this relation is greater or equal to one and bounded from above, so we continue as
\begin{eqnarray*}
\dots&=&\lsup s^\gi \log \tn(s,c_1T_{\tc_1})=c_1^\gi \lsup s^\gi\log\tn(s,T_1)\\
&=&c_1^\gi \ddg(T_1)=c^\gi_1.
\end{eqnarray*}
The computation for $\dg(T_\p)$ is completely analogous, and the proof is finished.
\end{proof}

\begin{lemma}\label{l2} For any $\p_1\in\lni$, we have
\begin{equation}\label{e401}
\ddg(T_\p)\le ||\p_1||_\lni^\gi.
\end{equation}
\end{lemma}
\begin{proof} Observe that inequality \eqref{e401} is homogeneous with respect to $||\p_1||_\lni^\gi$, so we can suppose $||\p_1||_\lni=1$ without loss of generality. Setting $F:=T_\po=P_+\po: L^2_a(\bd)\to L^2_a(\bd)$, we have
\begin{equation*}
\begin{split}
T_\p^*T_\p&=P_+\bar\p P_+\p P_+=P_+ \bar\po\bar\pon P_+\pon\po P_+\le P_+|\pon\po|^2 P_+\\
&\le ||\pon||^2_\lni P_+|\po|^2 P_+\le FF^*,
\end{split}
\end{equation*}
where we used repeatedly that $X^*YX\le ||Y||X^*X$ for bounded operators $X,Y$, and $Y=Y^*$. So we have that
$s^2_n(T_\p)=s_n(T_\p^*T_\p)\le s_n(FF^*)$, where $\{s_n(T_\p)\}_n$ and $\{s_n(FF^*)\}_n$ are singular values of operators $T_\p$ and $FF^*$, respectively. Consequently, 
$$
\tn(s,T_\p)\le\tn(\sqrt s, FF^*),
$$
and we continue as
\begin{eqnarray*}
\ddg(T_\p)&=&\lsup s^\gi\log\tn(s,T_\p)\le \lsup s^\gi\log\tn(\sqrt s,FF^*)\\
&=&\lsup s^{1/(2\g)}\log \tn(s, FF^*)=\dd_{2\g}(FF^*)=\ddg(F)=\ddg(T_\po)=1
\end{eqnarray*}
by the first claim of Lemma \ref{l1}. The proof is finished. \end{proof}

\nt
{\it Proof of Theorem \ref{t1}.} \ Let $\p_1:\bt\to\bc$ be a complex-valued continuous function. For any given $\ep>0$, choose a step function $\ti\p_1$ with the property
$$
||\p_1-\ti\p_1||_\lni <\ep,
$$
and set $\ti\p:=\ti\p_1\p_0$.
Lemma \ref{l2} says that
$$
\ddg(T_\p-T_{\ti\p})=\ddg(T_{\p-\ti\p})\le||\p_1-\ti\p_1||_\lni^\gi\le \ep^\gi.
$$
On the other hand, Proposition \ref{p1} implies
\begin{eqnarray*}
|\lpp\ddg(T_\p)\rpp^{\g/(1+\g)}-\lpp\ddg(T_{\ti\p})\rpp^{\g/(1+\g)}|&\le&\lpp\ddg(T_\p-T_{\ti\p})\rpp^{\g/(1+\g)}<\ep^{1/(1+\g)},\\
|\lpp\dg(T_\p)\rpp^{\g/(1+\g)}-\lpp\dg(T_{\ti\p})\rpp^{\g/(1+\g)}|&\le&\lpp\ddg(T_\p-T_{\ti\p})\rpp^{\g/(1+\g)}<\ep^{1/(1+\g)}.
\end{eqnarray*}
We finish the proof passing to the limit with respect $\ep\to 0$.
\hfill $\Box$
%%*** By closing the procedure one can prove the formula for $\p_1\in\lni$ ???  Let $\{\p_n\}$ be a sequence of spet functions converging uniformly to $\p$. What is the class of $\p$ obtained in this way? These are functions which are continuous except of a set of measure zero (= Riemann-integrable functions). I think we can mention this, but it's not worth of writing in details...***

\section{Proof of the asymptotic orthogonality of operators $T_{\tc_j}$ and $T_{\tc_k}, \ j\not =k$}\label{s4}
The purpose of this section is to prove the following theorem. Notice that in the case $\p_0(z)=(1-r)^\g, \ \g>0$, the proof of the similar result is rather simple.
\begin{theorem}\label{t20}
For $j\not=k, \ j,k=1,\dots,L$, the Toeplitz operators $T_{\tc_j}, T_{\tc_k}$, see \eqref{e101}, are asymptotically orthogonal, that is
$$
T_{\tc_k}^*T_{\tc_j}, T_{\tc_k}T_{\tc_j}^*\in \ss^0_{2\g}.
$$
\end{theorem}
%For the definition of operators  $T_{\tc_j},\ j=1,\dots, L$, see \eqref{e101}. 
By default, we assume that $j\not=k, \ j,k=1,\dots,L$ throughout this section. We present the argument for the operators $T_{\tc_k}^*T_{\tc_j}$, the reasoning for $T_{\tc_k}T_{\tc_j}^*$ is completely similar.

Since the proof of Theorem \ref{t20} in logarithmic case is rather involved, it is divided into a few steps.

\subsection{STEP 1. Toeplitz operators with symbols which are compactly supported in $\bd$}\label{s41}
Take a small $0<\d<1/2$ and define the characteristic function of $\{z: |z|\le 1-\d\}$,
\begin{equation*}
\chi_{00}(z)=
\left\{
\begin{array}{ll}
1,& |z|\le 1-\d\\
0,& |z|>1-\d
\end{array}
\right. ,
\quad \chi_{01}(z)=1-\chi_{00}(z),\quad z\in\bd.
\end{equation*}
Then, write $T_{\tc_j}$ as $T_{\tc_j}=T_{\cho\tc_j}+T_{\chn\tc_j}$. Since the $\mathrm{supp}\, (\cho\tc_j)$ is compact in $\bd$,  the singular values $\{s_n(T_{\cho\tc_j})\}_n$ %of $T_{\cho\tc_j}$
decay exponentially, that is, there is a constant $C=C(\d)>0$ such that
$$
s_n(T_{\cho\tc_j})\le C (1-\d)^{2n},
$$
see Zhu \cite[Chap. 7]{zh1}. This bound is almost trivial and it follows from the estimate of the quadratic form of the operator $(T_{\cho\tc_j}$ with the help of explicit expression \eqref{e0012} for the projection $P_+$. 
%Pushnitski \cite[Sect. 3.2]{pu1}. 
We have 
$$
T_{\tc_k}^*T_{\tc_j}=T_{\cho\tc_k}^*T_{\cho\tc_j}+T_{\cho\tc_k}^*T_{\chn\tc_j}+T_{\chn\tc_k}^*T_{\cho\tc_j}+T_{\chn\tc_k}^*T_{\chn\tc_j}.
$$
All operators $T_{\tc_j}, T_{\cho\tc_j}, T_{\chn\tc_j}$ are bounded and so the singular values of the first three operator products on the RHS of the above relation decay exponentially as well. That is, proving that $T_{\tc_k}^*T_{\tc_j}\in \ss^0_{2\g}$ is equivalent to saying that $T_{\chn\tc_k}^*T_{\chn\tc_j}\in \ss^0_{2\g}$.

\subsection{STEP 2. Products of Toeplitz operators with smooth integral kernel}\label{s42}
Second, let $B_0=T_{\chn\tc_k}^*T_{\chn\tc_j}$. Recalling point (1) of Proposition \ref{p3}, we wish to prove $B_0\in\css_p\subset \ss^0_{2\g}$ for some $p>0$.

The above conclusion is rather simple to obtain in the case when $j<k-1$ or $j>k+1$, and, as always, $j,k=1,\dots, L$. This means that the intervals $I_j$ and $I_k$ do not touch. 
\begin{lemma}\label{l3} Let $j<k-1$ or $j>k+1$. Then $T_{\chn\tc_k}^*T_{\chn\tc_j}\in\css_2$, the Hilbert-Schmidt class.
\end{lemma}

\begin{proof} The proof is well-known and quite simple, see Pushnitski \cite{pu1}. Set
$$
S_j:=\{z\in\bd: z=re^{i\th}, r\in (1-\d,1), \th\in I_j\},\quad j=1,\dots, L.
$$
By definition, we have $\chi_{S_j}=\chn\chi_j$ and $\ti\chi_{S_j}:=\chn\ti\chi_j$. Consequently,
$$
T_{\ti\chi_{S_k}}^* T_{\ti\chi_{S_j}}=P_+\ti\chi_{S_k} P_+\ti\chi_{S_j} P_+=P_+ (\ti\chi_{S_k} P_+\ti\chi_{S_j}) P_+.
$$
Since $||P_+||=1$, we prove that the operator $(\ti\chi_{S_k} P_+\ti\chi_{S_j})$ belongs to $\css_2$, and this will give the claim of the lemma by \eqref{e02} .
Indeed, the integral operator $(\ti\chi_{S_k} P_+\ti\chi_{S_j})$ can be written by \eqref{e0012} as
\begin{equation*}
\begin{split}
((\ti\chi_{S_k} P_+\ti\chi_{S_j})f)(w)&:=\int_\bd K(w,z) f(z)\, dA(z)\\
&:=\int_\bd \frac{\chi_{S_k}(w)\p_0(w)\chi_{S_j}(z)\p_0(z)}{(1-w\bar z)^2} f(z)\, dA(z),\quad f\in A^2(\bd).
\end{split}
\end{equation*}
Since the distance between regions $S_k$ and $S_j$ is strictly positive, the integral kernel $K(w,z)$ of the operator is a bounded function. Consequently, the kernel $K(.,.)$ of the operator lies in $L^2(\bd\times\bd; dA(w)\wedge dA(z)),$ and so the operator  $\ti\chi_{S_k} P_+\ti\chi_{S_j}$ is Hilbert-Schmidt, see \cite[Chap. 11, Sect. 3]{bs1}. The lemma is proved.
\end{proof}

\subsection{STEP 3. The case of $T_{\tc_k}^*T_{\tc_j}$ with neighboring intervals, $k=j\pm 1$; a bound on a kernel of an integral operator}\label{s43}
The arguments of the previous subsections show that it remains to prove that 
$B_0:=T_{\chn\tc_k}^*T_{\chn\tc_j}\in \ss^0_{2\g}$ for $k=j\pm1,\ j,k=1,\dots,L$. By rotation, assume WLOG that $k=1$ and $j=L$. To simplify the notation, we set $\d=2\pi/L$ and
\begin{eqnarray*}
\o_1&:=&\chn\tc_1=\chi_{S_1}\p_0, \quad S_1:=\lt\{z=re^{i\th}: r\in (1-\d,1),\th\in \lp0,\frac{2\pi}L\rp\rt\},\\
\o_L&:=&\chn\tc_L=\chi_{S_L}\p_0,\quad S_L:=\lt\{z=re^{i\th}: r\in (1-\d,1),\th\in \lp\frac{2\pi(L-1)}{L},2\pi\rp\rt\}.
\end{eqnarray*}
Alternatively, one can write the sets $S_1$ and $S_L$ as
\begin{equation*}
\begin{split}
&S_1:=\{z=re^{i\th}: r\in (1-\d,1),\th\in (0,\d)\},\\
&S_L:=\{z=re^{i\th}: r\in (1-\d,1),\th\in (-\d,0)\},
\end{split}
\end{equation*}
so that $S_L=\bar S_1$ by complex conjugation. We put $B_0=T^*_{\o_1}T_{\o_L}$, and
\begin{eqnarray}
0&\le&D_0:=B_0^*B_0=(T_{\o_1}^*T_{\o_L})^*(T_{\o_1}^*T_{\o_L})=T^*_{\o_L}T_{\o_1}T^*_{\o_1}T_{\o_L}\nn \\
&=&P_+\o_LP_+\o_1P_+\o_1P_+\o_L P_+\label{e11} 
%\le D:=\o_LP_+\o_1^2P_+\o_L \label{e11}\\
%&=&\chi_{S_L}\p_0P_+\p_0^2\chi_{S_1}P_+\p_0\chi_{S_L}. \label{e12}
\end{eqnarray}
Recalling inequality \eqref{e02}, we can omit two projections $P_+$ bordering the latter expression on the left and on the right. Consequently, the singular values (or eigenvalues) $s_n(D_0)$ are bounded from above by singular values (eigenvalues) $s_n(D)$ of the positive operator $D$ defined as
\begin{equation}\label{e12}
0\le\o_LP_+\o_1 P_+\o_1P_+\o_L\le D:=\o_LP_+\o_1^2P_+\o_L,
\end{equation}
where we used that $X^*YX\le ||Y||X^*X$ for two bounded operators $X$ and $Y, Y^*=Y$.

Now, we are interested in upper bounds on integral kernels for operators $D, D^m$ with a natural $m>0$. Let
\begin{eqnarray*}
(Df)(z_1)&:=&(D^1f)(z_1):=\int_{\bd, z_3} D_1(z_1,z_3)f(z_3)\, dA(z_3),\\
(D^m f)(z_1)&:=&\int_{\bd, z_{2m+1}} D_m(z_1,z_{2m+1})f(z_{2m+1})\, dA(z_{2m+1})\\
&=&\int_{\bd,z_{2m+1}}\lp\ldots\int_{\bd,z_3} \underbrace{D_1(z_1,z_3)\ldots D_1(z_{2m-1},z_{2m+1})}_{m}\rp
f(z_{2m+1})\\
&& \ \prod^m_{j=1}dA(z_{2j+1}), 
\end{eqnarray*}
the notation for indices $z_j$ will be made clear a bit later. When needed, we shall indicate the variable of the integration as a sub-index of the integral as it is written above.

\begin{proposition}\label{p5} We have
\begin{equation}\label{e16}
\begin{split}
&D_1(z_1,z_3)=\int_{\bd, z_2}\frac{\tcsl(z_1)\tcso(z_2)^2\tcsl(z_3)}{(1-z_1\bar z_2)^2(1-z_2\bar z_3)^2}\, dA(z_2),\\
&D_m(z_1,z_{2m+1})=\underbrace{\int_{\bd,z_2}\dots\int_{\bd,z_{2m}}}_{2m-1}
\frac{\prod_{j=1}^m \tcsl(z_{2j-1})\tcso(z_{2j})^2\tcsl(z_{2j+1})}
{\prod_{j=1}^m (1-z_{2j-1}\bar z_{2j})^2(1-z_{2j}\bar z_{2j+1})^2}\, \prod_{j=2}^{2m}dA(z_j).
\end{split}
\end{equation}
\end{proposition}

\begin{proof}
Relations \eqref{e11}, \eqref{e12} and the form of the Riesz orthoprojector $P_+$ in the Bergman space \eqref{e0012} give that, for $f\in L^2_a(\bd)$,
\begin{eqnarray}
Df(z_1)&=&\int_{\bd, z_2}\frac{\tcsl(z_1)\tcso(z_2)}{(1-z_1\bar z_2)^2}\, \lp\int_{\bd, z_3}\frac{\tcso(z_2)\tcsl(z_3)}{(1-z_2\bar z_3)^2}\, f(z_3) dA(z_3)\rp\, dA(z_2)\nn \\
&=&\int_{\bd, z_3} D_1(z_1,z_3)f(z_3)\, dA(z_3), \label{e13}
\end{eqnarray}
where, by Fubini-Tonelli theorem,
\begin{eqnarray}
D_1(z_1,z_3)&=&\int_{\bd, z_2}\frac{\tcsl(z_1)\tcso(z_2)^2\tcsl(z_3)}
{(1-z_1\bar z_2)^2(1-z_2\bar z_3)^2}\, dA(z_2). \label{e14}
\end{eqnarray}
This is the first equality in \eqref{e16}; the second one is proved similarly by an elementary induction.
\end{proof}
We stress that $z_{2j}\in S_1, \ j=1,\dots, m$, and $z_{2j+1}\in S_L, \ j=0,\dots, m$, in the second relation \eqref{e16}.
%*** explain the change of variables in $S_1$.

Before going to the proof of a coming proposition,we introduce some notation.  First, set
\begin{equation}\label{e15}
\psi_0(r):=\frac 1{\lp 1+1/\log r\rp^\g}, \quad 1/2<r<1.
\end{equation}
Set $\dh=\sqrt 2\d$ and define slightly different domains
\begin{eqnarray}\label{e150}
\shl&:=&\{z: |z|<1, |z-1|<\dh, \ \im z<0\} \\
&=&\{z=1-re^{i\th}: 0\le r<\min\{\dh, 2\cos\th\}, \ \th\in (0,\pi/2)\}, \nn
\end{eqnarray}
see Figure \ref{fig1}. It is clear that $S_L\subset \shl$, and so $\chi_{S_L}\le \chi_\shl$. The region $\sho$ is defined for $S_1$ in the same manner. 

Now, it is convenient to make the polar change of variables; notice that the new variables $(r_j,\th_j)$ are centered at the point $z_0=1$, and not the origin $z_0=0$. More precisely, we set
\begin{eqnarray}\label{e17}
&& z_{2j+1}:=1-r_{2j+1}e^{-i\th_{2j+1}}\in\sho, \ j=0,\dots,m,\\
&& z_{2j}:=1-r_{2j}e^{i\th_{2j}}\in\shl,  \ j=1,\dots,m.\nn
\end{eqnarray}
We have $z_{2j+1}\in\sho$ and $z_{2j}\in\shl$ for $r_k\in (0,\sqrt2\d), \th_k\in (0,\pi/2),\ k=1,\dots, 2m+1$.

\begin{figure}[tbp]
%\vspace{2cm}
\includegraphics[width=14cm]{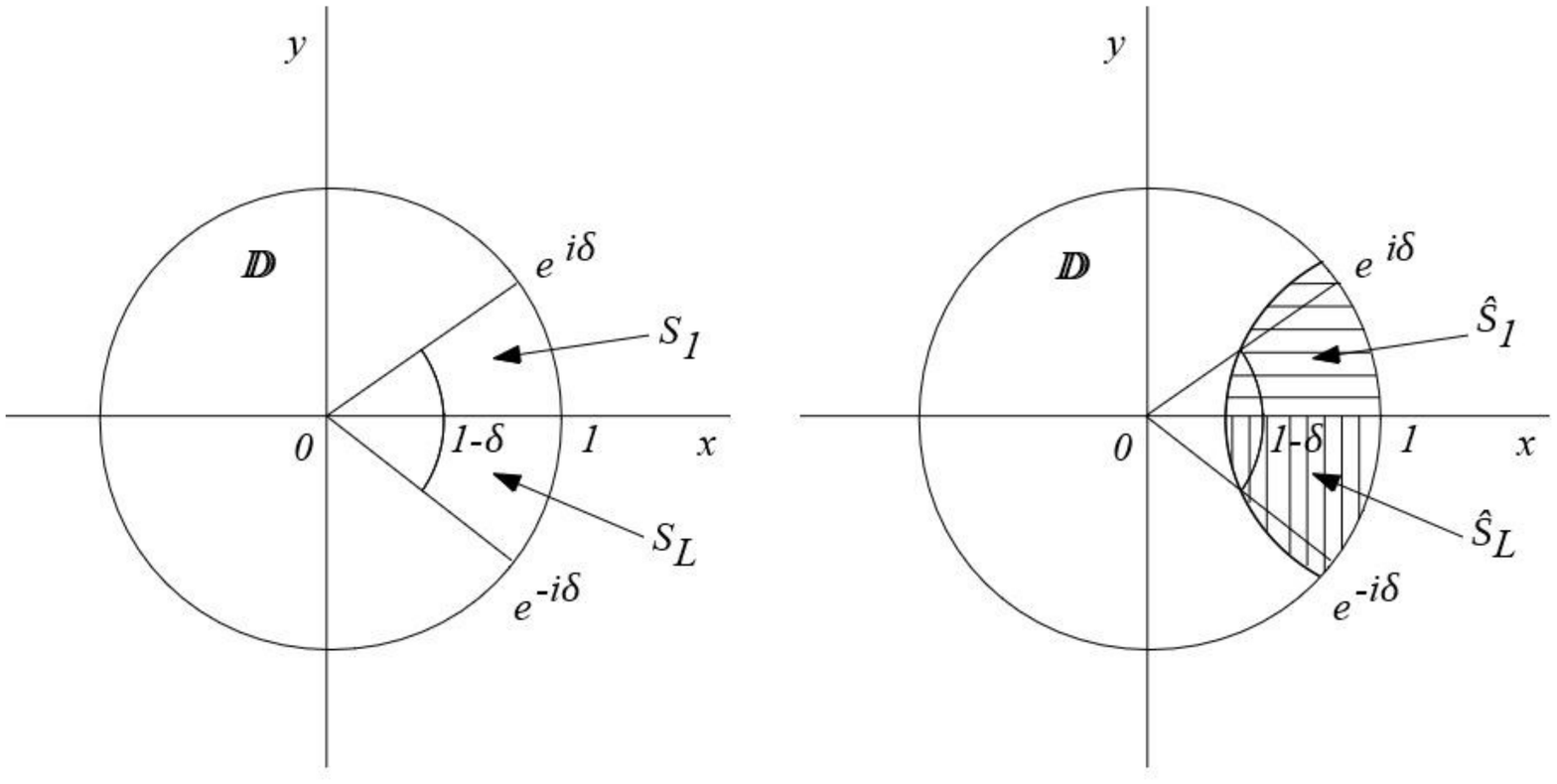}
\vspace{-2cm}
\caption{Domains $S_1, S_L$, and $\hat S_1, \hat S_L$.}
\label{fig1}
\end{figure}

\begin{proposition}\label{p6} We have the following bound on kernels $D_m$ \eqref{e16} in terms of the variables introduced in \eqref{e17}
\begin{equation}\label{e18}
\begin{split}
&|D_m(z_1,z_{2m+1})|\\
&\le 
C\underbrace{\int^\dh_{0,r_2}\dots\int^\dh_{0,r_{2m}}}_{2m-1}
\frac{\chi_\shl(z_1)\psi_0(r_1)\, \lp\prod^{2m}_{j=2}\psi_0(r_j)^2\rp\, \chi_\shl(z_{2m+1})\psi_0(r_{2m+1})}
{\prod^{2m}_{j=1}(r_j+r_{j+1})}\, \prod^{2m}_{j=2}dr_j.
\end{split}
\end{equation}
%where we put $z_1\in\sho\mapsto \bar z_1\in\shl$ and $z_{2m+1}\in\sho\mapsto \bar z_{2m+1}\in\shl$ for the simplicity of writing.
\end{proposition}

\begin{proof}
First, we take the modulus under the integral in the second relation \eqref{e16}. Since $S_1\subset\sho$ and $S_L\subset\shl$, we see
$\tc_{S_1}\le \tc_\sho$ and $\tc_{S_L}\le \tc_\shl$, so we can replace $\tc_{S_1}, \tc_{S_L}$ with $\tc_\sho,\tc_\shl$
under the integral. Second, to uniformize the notation, we make the change of variables
$$
z_{2j}\in\sho \mapsto \bar z_{2j}\in\shl.
$$
In this way, we have that $z_k=1-r_ke^{i\th}\in\shl$ for all $k=2,\dots, 2m$.

Furthermore, we get with the help of the new variable $z_{2j}$
$$
1-z_{2j-1}z_{2j}=r_{2j-1}e^{i\th_{2j-1}}+r_{2j}e^{i\th_{2j}}-r_{2j-1}r_{2j}e^{i(\th_{2j-1}+\th_{2j})}.
$$
Consequently, there is a $C=C(\dh)>0$ such that
$$
|1-z_{2j-1}z_{2j}|\ge C(r_{2j-1}+r_{2j}).
$$
Remind that $\tc_\shl(z)=\chi_\shl(z)\p_0(|z|)$. Concerning these factors, we notice that $1-|z_k|\le |1-z_k|=r_k, \ k=2,\dots, 2m$, and so
$$
\p_0(z_k)\le \psi_0(r_k).
$$
Plugging all these bounds into the second relation \eqref{e16}, we obtain
\begin{equation*}
\begin{split}
&|D_m(z_1,z_{2m+1})|\le\underbrace{\int_{\bd,z_2}\dots\int_{\bd,z_{2m}}}_{2m-1}
\frac{\prod_{j=1}^m \tc_\shl(z_{2j-1})\tc_\shl(z_{2j})^2\tc_\shl(z_{2j+1})}
{\prod_{j=1}^m |1-z_{2j-1} z_{2j}|^2 |1-z_{2j}z_{2j+1}|^2}\, \prod_{j=2}^{2m}dA(z_j)\\
&\le C \underbrace{\int^{\pi/2}_{0,\th_2}\int^\dh_{0,r_2}r_2\, dr_2d\th_2\, \dots
\int^{\pi/2}_{0,\th_{2m}}\int^\dh_{0,r_{2m}}r_{2m}\, dr_{2m}d\th_{2m}}_{2m-1}
\frac{\tc_\shl(z_1)\, \lp\prod_{j=2}^{2m}\psi_0(r_j)^2\rp\, \tc_\shl(z_{2m+1}) }
{\prod_{j=1}^{2m} (r_j+r_{j+1})^2}\\
&\le C \lp\frac\pi 2\rp^{2m-1} \underbrace{\int^\dh_0 \dots\int^\dh_0}_{2m-1}
\frac{\tc_\shl(z_1)\, \lp\prod_{j=2}^{2m}\psi_0(r_j)^2\rp\, \tc_\shl(z_{2m+1}) }
{\prod_{j=1}^{2m} (r_j+r_{j+1})}\, \prod_{j=2}^{2m}dr_j,
\end{split}
\end{equation*}
where we used that $r_{j+1}/(r_j+r_{j+1})^2\le 1/(r_j+r_{j+1})$. The proposition is proved.
\end{proof}

\subsection{STEP 4. The case of $T_{\tc_k}^*T_{\tc_j}$ with $k=j\pm 1$; an auxiliary combinatorial lemma and the final computation}\label{s44}
The integral from the RHS of \eqref{e18} is computed on a cube $(r_2,\dots, r_{2m})\in \ccs=(0,\dh)^{2m-1}$. Roughly speaking, the main idea for the calculation of this subsection is to divide the cube in standard simplexes
$$
\ccs_{i_2,\dots,i_{2m}}=\{(r_2,\dots,r_{2m})\in\ccs: r_{i_2}\in (0,\dh),\ \dh>r_{i_2}>\dots>r_{i_{2m}}>0\},
$$
and to obtain an appropriate bound for \eqref{e18} from above integrating on every simplex. Here, $\{i_j\}_{j=2,3,\ldots, 2m}$ is a transposition of the set $\{2,3,\ldots, 2m\}$. %Here, $i_j,i_k=2,3,\dots,2m$, and $i_j\not=i_k$.

\begin{lemma}\label{l4} Let $r_j\ge 0, \ j=1,\dots,2l$, where $l>0$ is a natural number.  Then
\begin{equation}\label{e19}
\begin{split}
&(r_1+r_2)\lp \prod^{l-1}_{j=2}(r_j+r_{j+1})\rp\, (r_l+r_{l+1})\cdot
(r_1+r_{l+2})\lp \prod^{2l-1}_{j=l+2}(r_j+r_{j+1})\rp (r_{2l}+r_{l+1})\\
&\qquad \ge (\max_{i=1,\dots, 2l} r_i)^2\lp \prod^{2l}_{j=1}r_j\rp'', 
\end{split}
\end{equation}
where $(\prod .)''$ means that we drop both the maximal and the minimal factors in the product.
\end{lemma}

Suppose that we have $r_{i_1}> r_{i_2}> \dots r_{i_{2l}}>0$. Then the lemma says that
\begin{equation*}%\label{}
\begin{split}
&(r_1+r_2)\lp \prod^{l-1}_{j=2}(r_j+r_{j+1})\rp\, (r_l+r_{l+1})\cdot
(r_1+r_{l+2})\lp \prod^{2l-1}_{j=l+2}(r_j+r_{j+1})\rp\, (r_{2l}+r_{l+1})\\
&\qquad \ge r_{i_1}^2 \lp\prod^{2l-1}_{j=2}r_{i_j}\rp.
\end{split}
\end{equation*}

\begin{proof}
We can assume that $r_j>0, \ j=1,\dots, 2l$. Consider the decomposition of the set $\{r_1,\ldots,r_{2l}\}$ into two disjoint subsets $G_1, G_2$, where
$$
G_1=\{r_1,r_{l+1}\}, \quad G_2=\{(r_j)_{j=2,\dots, l}, \ (r_j)_{j=l+2,2l}\}.
$$
Using the symmetry of variables $r_j$ within each group, one can reduce the general situation to the analysis of four cases only: 1)
$\max_{j=1,\dots, 2l} r_j\in G_1,\ \min_{j=1,\dots, 2l} r_j\in G_1$; 2) $\max_{j=1,\dots, 2l} r_j\in G_1,\ \min_{j=1,\dots, 2l} r_j\in G_2$; 3) $\max_{j=1,\dots, 2l} r_j\in G_2,$ $\min_{j=1,\dots, 2l}$ $r_j\in G_1$; 4) 
$\max_{j=1,\dots, 2l} r_j\in G_2, \min_{j=1,\dots, 2l} r_j\in G_2$.

\emph{Case 1:} Without loss of generality, suppose $r_1=\max_{j=1,\dots, 2l} r_j\in G_1,$ $r_{l+1}=$ $\min_{j=1,\dots, 2l} r_j\in G_1$. We have
\begin{eqnarray*}
&&(r_1+r_2)(r_2+r_3)\dots (r_l+r_{l+1})\cdot (r_1+r_{l+2})(r_{l+2}+r_{l+3})\dots (r_{2l}+r_l)\\
&&\ge r_1 r_2\dots r_l\cdot r_1 r_{l+2}\dots r_{2l}=r_1^2\, \prod^{2l}_{j=2, j\not=l+1} r_j.
\end{eqnarray*}

\emph{Case 2:} WLOG $r_1=\max_{j=1,\dots, 2l} r_j\in G_1,\ r_2=\min_{j=1,\dots, 2l} r_j\in G_2$. %Of course, the choice of $r_2$ does not matter, and the bound is the same for any $r_k, k\in G_2$. 
Then
\begin{eqnarray*}
&&(r_1+r_2)(r_2+r_3)\dots (r_l+r_{l+1})\cdot (r_1+r_{l+2})(r_{l+2}+r_{l+3})\dots (r_{2l}+r_l)\\
&&\ge r_1 r_3\dots r_l r_{l+1}\cdot r_1 r_{l+2}\dots r_{2l}=r_1^2\, \prod^{2l}_{j=3} r_j.
\end{eqnarray*}

\emph{Case 3:} WLOG $r_2=\max_{j=1,\dots, 2l} r_j\in G_2, \ r_1=\min_{j=1,\dots, 2l} r_j\in G_1$. We have
\begin{eqnarray*}
&&(r_1+r_2)(r_2+r_3)\dots (r_l+r_{l+1})\cdot (r_1+r_{l+2})(r_{l+2}+r_{l+3})\dots (r_{2l}+r_l)\\
&&\ge r_2 r_2 r_3\dots r_l\cdot r_{l+2} r_{l+3}\dots r_{2l} r_{l+1}=r_2^2\, \prod^{2l}_{j=3} r_j.
\end{eqnarray*}

\emph{Case 4:} WLOG $r_2=\max_{j=1,\dots, 2l} r_j\in G_2,\ r_3=\min_{j=1,\dots, 2l} r_j\in G_2$. We get
\begin{eqnarray*}
&&(r_1+r_2)(r_2+r_3)\dots (r_l+r_{l+1})\cdot (r_1+r_{l+2})(r_{l+2}+r_{l+3})\dots (r_{2l}+r_l)\\
&&\ge r_2 r_2 r_4\dots r_{l+1}\cdot r_1 r_{l+2}\dots r_{2l}=r_2^2\, \prod^{2l}_{j=1, j\not=2,3} r_j.
\end{eqnarray*}
The proof is finished.
\end{proof}

\begin{proposition}\label{p7} The following assertions take place:
\begin{enumerate}
\item We have
\begin{equation}\label{e20}
\begin{split}
&||D^m||^2_{\css_2}\le C\underbrace{\int^\dh_0\dots\int^\dh_0}_{4m}
\frac{\prod^{4m}_{j=1} \psi_0(r_j)^2}{(r_1+r_2)^2 \lp\prod^{2m-1}_{j=2}(r_j+r_{j+1})\rp (r_{2m}+r_{2m+1})}\\
&\hspace{2,5cm} \frac 1{(r_1+r_{2m}) \lp\prod^{4m-1}_{j=2m+2}(r_j+r_{j+1})\rp (r_{4m}+r_{2m+1})}\, \prod^{4m}_{j=1}dr_j.
\end{split}
\end{equation}
\item In particular, $D^m\in \css_2$ (or, equivalently, $D\in \css_{2m}$) whenever $\g>1/(8m)$.
\end{enumerate}
\end{proposition}

\begin{proof}
The proof of the proposition is lengthy, but rather elementary. For the convenience of the reader, we first treat the case $m=1$, and then go to the general $m$.

\emph{First}, let $m=1$. Recall the notation introduced in Proposition \ref{p6}. Relation \eqref{e18} says that
\begin{equation}\label{e201}
|D(z_1,z_3)|\le C\int^\dh_0 \frac{\tc_\shl(z_1)\psi_0(r_2)^2\tc_\shl(z_3)}{(r_1+r_2)(r_2+r_3)}\, dr_2,\quad z_1,z_3\in\bd.
\end{equation}
 We continue as
\begin{eqnarray}
I_1&:=&||D||^2_{\css_2}=\int_{\bd, z_1}\int_{\bd, z_3} |D(z_1,z_3)|^2\, dA(z_1)dA(z_3) \nn\\
&=&\int_{\shl, z_1}\int_{\shl, z_3} |D(z_1,z_3)|^2\, dA(z_1)dA(z_3). \label{e21}
\end{eqnarray}
Now, we pass to the polar coordinates translated to point $z_0=1$
$$
z_1=1-r_1e^{i\th_1}, \ z_3=1-r_3e^{i\th_3}
$$
exactly as we did in Proposition \ref{p6}. Here, $r_j\in (0,\dh),\ \th_j\in (0,\pi/2), \ j=1,3$. We replace each factor $|D(z_1,z_3)|$ in the integal on RHS of \eqref{e21} by the integral bound given in \eqref{e18}. Hence, the quantity in the RHS of \eqref{e21} is written as a 4-tuple integral, and it is estimated above as
\begin{equation}\label{e22}
I_1 \le C\underbrace{\int^\dh_{0,r_1}\dots \int^\dh_{0,r_4}}_{4}
\frac{\prod_{j=1}^4 \psi_0(r_j)^2}{(r_1+r_2)(r_2+r_3)\cdot (r_1+r_4)(r_4+r_3)}\, \prod^4_{j=1}dr_j.
\end{equation}
This is point (1) of the proposition for $m=1$.

As already mentioned, we cut the cube $\ccs=\ccs(\dh)=(0,\dh)^4$ in simplexes 
$$
\ccs_{i_1\dots i_4}=\{(r_1,r_2,r_3,r_4): \dh>r_{i_1}>\dots r_{i_4}>0\},
$$
where $i_k=1,\dots, 4, i_k\not=i_j, k\not=j, \ k,j=1,\dots,4$. Lemma \ref{l4} with $l=2$ says trivially that
$$
(r_1+r_2)(r_2+r_3)\cdot (r_1+r_4)(r_4+r_3)\ge r_{i_1}^2 r_{i_2}r_{i_3}.
$$ 
Consequently, the upper bound for integral \eqref{e22} on the simplex $\ccs_{i_1\dots i_4}$ reads as
\begin{eqnarray}\label{e23}
&&\int_{\ccs_{i_1\dots i_4}}\frac{\prod_{j=1}^4 \psi_0(r_j)^2}{(r_1+r_2)(r_2+r_3)\cdot (r_1+r_4)(r_4+r_3)}\, \prod^4_{j=1}dr_j \\
&\le& C \int^\dh_0 dr_{i_1}\, \frac{\psi_0(r_{i_1})^2}{r^2_{i_1}}\, \int^{r_{i_1}}_0 dr_{i_2}\,\frac{\psi_0(r_{i_2})^2}{r_{i_2}} \dots \int^{r_{i_3}}_0 dr_{i_4}\, \psi_0(r_{i_4})^2 \nn\\
&\le& C \int^\dh_0 dr_{i_1}\, \frac{\psi_0(r_{i_1})^2}{r^2_{i_1}}\, \int^{r_{i_1}}_0  dr_{i_2}\, \frac{\psi_0(r_{i_2})^2}{r_{i_2}}\, \int^{r_{i_2}}_0 dr_{i_3}\, \frac{\psi_0(r_{i_3})^4 r_{i_3}}{r_{i_3}} \nn\\
&\le& C \int^\dh_0 dr_{i_1}\, \frac{\psi_0(r_{i_1})^2}{r^2_{i_1}}\, \int^{r_{i_1}}_0 dr_{i_2}\, \frac{\psi_0(r_{i_2})^6 r_{i_2}}{r_{i_2}} \nn\\
&\le& C \int^\dh_0 dr_{i_1}\, \frac{\psi_0(r_{i_1})^8 r_{i_1}}{r^2_{i_1}}=C \int^\dh_0 dr_{i_1}\, \frac{\psi_0(r_{i_1})^8}{r_{i_1}} \nn
\end{eqnarray}
where we used that $\psi_0(r_{i+1})\le \psi_0(r_i)$ since $r_{i+1}\le r_i$ and the function $\psi_0$ is increasing.
The latter integral is convergent whenever $8\g>1$, which is point (2) of the current proposition with $m=1$.

\emph{Second}, we follow the lines of the proof for $m=1$ in the general case, but we use some combinatorics. We get
\begin{eqnarray*}
I_m&:=&||D^m||^2_{\css_2}=\int_{\bd, z_1}\int_{\bd, z_{2m+1}} |D_m(z_1,z_{2m+1})|^2\, dA(z_1)dA(z_{2m+1}) \nn\\
&=&\int_{\shl, z_1}\int_{\shl, z_{2m+1}} |D_m(z_1,z_{2m+1})|^2\, dA(z_1)dA(z_{2m+1}). %\label{e21}
\end{eqnarray*}
Now, we bound $|D(z_1,z_{2m+1})|^2$ from above by the product of expressions from \eqref{e18}. So, we come to a $4m$-tuple integral
\begin{eqnarray*}
&&I_m\le C\underbrace{\int^\dh_0\dots\int^\dh_0}_{4m}
\frac{\prod^{4m}_{j=1} \psi_0(r_j)^2}{(r_1+r_2) \lp\prod^{2m-1}_{j=2}(r_j+r_{j+1})\rp (r_{2m}+r_{2m+1})}\\
&&\qquad\qquad \frac 1{(r_1+r_{2m}) \lp\prod^{4m-1}_{j=2m+2}(r_j+r_{j+1})\rp (r_{4m}+r_{2m+1})}\, \prod^{4m}_{j=1}dr_j,
\end{eqnarray*}
the computation being completely analogous to \eqref{e22}. Point (1) of the proposition is proved in the general case.

We now divide the cube $\ccs=\ccs(\dh)=(0,\dh)^{4m}$ in simplexes 
$$
\ccs_{i_1\dots i_{4m}}=\{(r_1,\dots,r_{4m}): \dh>r_{i_1}>\dots >r_{i_{4m}}>0\},
$$
where $i_k=1,\dots, 4m, i_k\not=i_j, k\not=j, \ k,j=1,\dots, 4m$. Lemma \ref{l4} ($l=2m$) gives
\begin{equation*}
\begin{split}
&\lp \prod^{2m}_{j=1}(r_j+r_{j+1})\rp \cdot
(r_1+r_{2m+2})\lp \prod^{4m-1}_{j=2m+2}(r_j+r_{j+1})\rp (r_{4m}+r_{2m+1})\\
&\ge r^2_{i_1}\lp\prod^{4m-1}_{j=2} r_{i_j}\rp.
\end{split}
\end{equation*}
Consequently, the upper bound for the integral on $\ccs_{i_1\dots i_{4m}}$ is
\begin{eqnarray*}
&&\int_{\ccs_{i_1\dots i_{4m}}}\dots \le 
C \int^\dh_0 dr_{i_1}\, \frac{\psi_0(r_{i_1})^2}{r^2_{i_1}}\, \int^{r_{i_1}}_0 dr_{i_2}\, \frac{\psi_0(r_{i_2})^2}{r_{i_2}}\, \dots \int^{r_{i_{4m-1}}}_0 dr_{i_{4m}}\, \psi_0(r_{i_{4m}})^2.
\end{eqnarray*}
As before, we have $\psi_0(r_{i+1})\le \psi_0(r_i)$ since $\psi_0$ is increasing and $0<r_{i+1}\le r_i$. We bound this integral by telescoping as in \eqref{e23} to obtain
\begin{eqnarray*}
&& \int_{\ccs_{i_1\dots i_{4m}}}\dots \le C \int^\dh_0 dr_{i_1}\, \frac{\psi_0(r_{i_1})^{8m} r_{i_1}}{r^2_{i_1}}=C \int^\dh_0 dr_{i_1}\, \frac{\psi_0(r_{i_1})^{8m}}{r_{i_1}}.
\end{eqnarray*}
The latter integral is convergent whenever $8m\g>1$, which is point (2) of the current proposition. The proof is finished.
\end{proof}
%\begin{corollary}\label{c1}
%Let $\g>0$ be fixed. It follows that $T^*_{\o_1}T_{\o_L}\in \ss^0_{2\g}$, see Section \ref{s43} for notation.
%\end{corollary}

\nt
{\it Proof of Theorem \ref{t20}.} 
Indeed, fix a natural $m>0$ such that $\g>1/(8m)$. Recall that
$$
(T^*_{\o_1}T_{\o_L})^*(T^*_{\o_1}T_{\o_L})\le D,
$$
and, by Proposition \ref{p7}, the compact positive operator $D$ lies in $\css_{2m}\subset \ss^0_{2\g}$. By monotonicity of singular values, the operator  $T^*_{\o_1}T_{\o_L}$ is in $\ss^0_{2\g}$ as well, and the theorem is proved. \hfill$\Box$

\section{Applications and concluding remarks}\label{s5}
The corollaries presented in this Section mainly follow Pushnitski \cite{pu1}. Once again, remind the notation for functions $\p_0,\p_1, \p$, \eqref{e0041}, \eqref{e0042}.

\subsection{Different verisons of the obtained results}\label{s51}

\begin{corollary}\label{c1}\hfill

\begin{enumerate}
\item Let $\p$ be a function continuous on $\bar\bd$ and having the property that
$$
\lim_{|z|\to1-0} |\p_{0,\g}(|z|)^{-1}\p(z)-\p_1(e^{i\th})|=0, \quad z=|z|e^{i\th}\in\bd,
$$
for some $\p_1\in C(\bt)$. Then the singular values of operator $T_\p$ have the asymptotics \eqref{e005}, \eqref{e0051}.
\item In particular, let 
$$\p(z)=\p_1(e^{i\th})\p_{0,\g}(|z|)\, g(|z|),$$
where $g\in L^\infty[0,1)$ and  $g(1):=\lim_{r\to1-0} g(r)$. Then 
$$
\lim_{n\to+\infty} \log(n+1)^\g s_n(T_\p)=\lp |g(1)|\,||\p_1||_{L^\infty(\bt)}\rp^\g,
$$
idem for the counterpart of formula \eqref{e0051}.
\end{enumerate}
\end{corollary}
The proof of the first claim of the corollary is similar to the reasoning of Section 3 and it is omitted. The second claim is an easy consequence of the first one. 
%For the second claim, we have to use Lemma \ref{l01} with the prescribed function $g$ instead of $g(r)\equiv 1$ for $r\in [0,1]$, as it was done in Theorem \ref{t1}.

The general operator-theoretic techniques developed in Pushnitski-Yafaev \cite{puy} permit one to treat the sequences of positive and negative eigen-values $\{\l^\pm_n(T_\p)\}_n$ of operator $T_\p$ with a real symbol $\p$. Recall that $\p^\pm_1=\max\{\pm\p_1,0\}$. For simplicity, we give the following corollary for a symbol $\p$ given in \eqref{e0041}, \eqref{e0042}.
\begin{corollary}\label{c2}
Let $\p$ be a real symbol as above. We have the following asymptotics for the sequences of positive and negative eigenvalues of operator $T_\p$, respectively
$$
\lim_{n\to+\infty} \log(n+1)^\g \l^\pm_n(T_\p)=||\p^\pm_1||_{L^\infty(\bt)}^\g.
$$
\end{corollary}

\begin{corollary}\label{c3}
Formulae \eqref{e005}, \eqref{e0051} hold for (finitely) piece-wise continuous functions $\p_1$ defined on $\bt$.
\end{corollary}
The last corollary follows from the fact that the piece-wise continuous finctions lie in the closure of step-functions on $\bt$ in $L^\infty(\bt)$-norm. 

\subsection{An application to finite-banded matrices}\label{s52}
Clearly enough, the obtained results allow us to handle the singular values of finite-banded matrices with loga\-rithmically decaying entries. We give a couple of definitions to make this more precise.

Let $D:\ell^2(\bz_+)\to \ell^2(\bz_+)$ be a compact operator. Its matrix is denoted by
$$
D:=[d_{i,j}]_{i,j=0,\dots,\infty}
$$
in the standard basis of $\ell^2(\bz_+)$. For a fixed $N\in\bn$, we say that $D$ is a $(2N+1)$-banded matrix, if $d_{i,j}=0$ for $|i-j|>N+1$. Furthermore, let $\{b_{-N},\dots, b_{N}\}$ be a set of $2N+1$ complex coefficients. For a fixed $\g>0$, we say that a $(2N+1)$-banded matrix $D$ has logarithmically decaying entries, if
\begin{equation*}%\label{}
d_{m,m+j}=\frac{b_j}{(\log m)^\g}(1+o(1)),\ \quad m\to +\infty, \ j=-N,\dots,N.
\end{equation*}
To give the next corollary, define a specific function $\p_1$ corresponding to coefficients $\{b_j\}_{j=-N,\dots,N}$ as
$$
\p_{1,b}(e^{i\th}):=\sum^N_{j=-N} b_je^{i j\th}.
$$
\begin{corollary}\label{c4} We have the following asymptotics for the singular values of the above $(2N+1)$-banded matrix $D$
\begin{equation}\label{e25}
\lim_{n\to+\infty} (\log n)^\g s_n(D)=||\p_{1,b}||_{L^\infty(\bt)}.
\end{equation}
\end{corollary}
The proof of the corollary follows at once from two facts. First, the singular values of Toeplits operator $T_\p$ with $\p=\p_{1,b}\p_0$ have asymptotics \eqref{e25}. Second, we write the operator $T_\p$ in the standard basis of the Bergman space $L^2_a(\bd)$ and we identify it with the obtained matrix on $\ell^2(\bz_+)$. It is easy to see now that $D-T_\p\in \ss^0_\g$, and the claim of the corollary follows from Proposition \ref{p2}. 

\section{Appendix A: Asymptotics of a logarithmic integral}\label{s6}
The following lemma is an easy consequence of a Watson-type lemma for Laplace integrals proved in Kupin-Naboko \cite[Thm. 0.2, Cor. 2.4]{kn1}.

\begin{lemma}\label{l01} Let $g\in L^\infty[0,1)$, and $g(1):=\lim_{r\to 1-0}g(r)\not=0$. Then, for a  $\g>0$,
$$
\int^1_0 \frac {r^n}{(1+\log1/(1-r))^\g}g(r)\, dr=\frac{g(1)}{n(\log n)^\g}\lp 1+o(1)\rp.
$$
\end{lemma}

\bigskip
\nt{\bf Acknowledgments.}  The authors would like to thank Omar El-Fallah, Leonid Golinski, Karim Kellay and Alexander Pushnitski for heplful discussions.

The work is partially supported by the project ANR-18-CE40-0035. 

S. Naboko kindly acknowledges the support by RScF-20-11-20032 grant and Knut and Alice Wallenberg Foundation grant. A part of this research was done during S. Naboko's visit to University of Bordeaux in October-November, 2019. He is grateful to the University for the hospitality. 

B. Touré gratefully acknowledges the financial support coming from agreements between University of Bordeaux and the Embassy of France in Mali, 2018.

\end{document}